\documentclass[11pt]{article}
\usepackage[utf8]{inputenc}
\usepackage{lmodern}
\usepackage{subfiles}
\usepackage{enumitem}
\setenumerate{topsep=6pt,ref={\normalfont(\roman*)},label={\normalfont(\roman*)}, itemsep=0pt}
\usepackage{pgfplots}
\pgfplotsset{compat=1.18}
\usepgfplotslibrary{groupplots}
\usepackage{amsfonts}
\usepackage{amsthm}
\usepackage{amsmath}
\usepackage{amssymb}
\usepackage{amscd}
\usepackage{mathrsfs}
\usepackage{mathtools}
\usepackage{bbm}
\usepackage{esint}

\usepackage[margin=3cm]{geometry}
\usepackage{setspace}
\usepackage{indentfirst}
\usepackage{graphicx}
\usepackage{graphics}
\usepackage{lscape}
\usepackage{pgf,tikz}
\usepackage{tikz-cd}
\usepackage{color}
\usepackage{pict2e}
\usepackage{epic}
\usepackage{epstopdf}
\usepackage{titlesec, titlefoot}
\titleformat{\section}[block]{\Large\bfseries\filcenter}{\thesection}{1em}{}
\usepackage{commath}
\usepackage{float}
\usepackage{caption}
\usepackage{etoolbox}
\usepackage[affil-it]{authblk}
\usepackage{combelow}

\usepackage[hidelinks, bookmarksdepth=3]{hyperref}
\hypersetup{bookmarksopen=true} 
\usepackage{hypcap}

\graphicspath{{./Pictures/}}
\allowdisplaybreaks

\expandafter\def\expandafter\normalsize\expandafter{%
\normalsize
\setlength\abovedisplayskip{6pt}
\setlength\belowdisplayskip{6pt}
\setlength\abovedisplayshortskip{6pt}
\setlength\belowdisplayshortskip{6pt}
}

\theoremstyle{plain}

\renewcommand*\thesection{\arabic{section}}
\numberwithin{equation}{section} 

\newtheorem{theorem}{Theorem}[section]
\newtheorem{lemma}[theorem]{Lemma}
\newtheorem*{lemma*}{Lemma}
\newtheorem{proposition}[theorem]{Proposition}

\theoremstyle{definition}

\newtheorem{remark}[theorem]{Remark}

\expandafter\let\expandafter\oldproof\csname\string\proof\endcsname
\let\oldendproof\endproof
\renewenvironment{proof}[1][\proofname]{%
\oldproof[\upshape \bfseries #1]%
}{\oldendproof}

\makeatletter
\def\@makechapterhead#1{%
\vspace*{50\p@}%
{\parindent \z@ \raggedright \normalfont
\interlinepenalty\@M
\Huge\bfseries  \thechapter.\quad #1\par\nobreak
\vskip 40\p@
}}
\makeatother

\newcommand{\eps}{\varepsilon}

\DeclareMathOperator{\ddiv}{div}

\DeclareMathOperator{\SO}{SO}

\def \Om{\Omega}
\def \R {\mathbb{R}}

\def \d{\,\textup{d}}

\def \p{\partial}
\def \mc{\mathcal}
\def \mb{\mathbb}

\def \Id{\textup{Id}}

\newcommand{\di}{\mathrm{div}}

\begin{document}

	\title{\textbf{Nonconstant weakly harmonic maps with constant trace}}
	
	\author[1]{{\Large David Ziener}}
	
	\affil[1]{\small Max Planck Institute for Mathematics in the Sciences, Inselstrasse 22, 04103 Leipzig, Germany
	\protect\\
	{\tt{david.ziener@mis.mpg.de}}\ }
	
	\date{}
	
    \maketitle

	\begin{abstract}
    We construct a nonconstant weakly harmonic map from a three-dimensional ball into the two-sphere with constant trace. This gives an affirmative answer to a question raised by Rivière in 1995.
    \end{abstract}

    \medskip


    \smallskip

	
	\unmarkedfntext{
	\hspace{-0.75cm}
	\emph{Acknowledgments.} DZ would like to thank his PhD advisor László Székelyhidi for his continued support and fruitful discussions regarding this work.
	}
	
	\vspace{1cm}

    \section{Preliminaries}
    
    For an open set $\Om \subset \mb R^d$, where in this paper $d \in \{2,3\}$, we denote
    \[
    H^1(\Om;\mb S^2):=\{u \in H^1(\Om;\mb R^3):|u|=1,\;\text{a.e.}\},\quad E_\Om(u)=\int_{\Om}|\nabla u|^2dx.
    \]
    A map $u \in H^1(\Om;\mb S^2)$ is weakly harmonic if
    \begin{equation}
        \label{eq:weakHarm}
        \int_{\Om}\nabla u \cdot \nabla \phi dx=\int_{\Om}|\nabla u|^2 u \cdot \phi dx\quad \forall \phi \in C_c^\infty(\Om;\mb R^3).
    \end{equation}
    Since the beautiful construction due to Rivi\`ere \cite{ever}, it is known that, in dimension three, weakly harmonic maps need not possess any partial regularity: In his paper Rivi\`ere constructs, for any smooth nonconstant boundary data on the boundary of a ball, a weakly harmonic map which is discontinuous everywhere. His construction is based on the relative relaxed energy as introduced by Bethuel, Brezis and Coron in \cite{BBC}. The relative relaxed energy turned out to be a useful tool for constructing harmonic maps with singularities. It was used by Hardt, Lin and Poon \cite{HLP} to construct axially symmetric harmonic maps with prescribed singularities on the axis. In \cite{Rline} Rivi\`ere iterated this construction to find harmonic maps with a line of singularities and finally, it played an important role in his celebrated construction \cite{ever}. We give the very rough idea behind the role of the relative relaxed energy in these proofs.\\
    On a smooth bounded domain $\Om \subset \mb R^3$, and for a fixed reference map $v \in H^1(\Om;\mb S^2)$ with smooth trace, the relaxed energy is given as
    \[
    F_v(u):=E_\Om(u)+8\pi L(u,v),\quad u \in H^1(\Om;\mb S^2),\quad u-v \in H^1_0(\Om;\mb R^3),
    \]
    where $L(u,v)$ denotes the length of the minimal connection associated with the difference of the topological singularities of $u$ and $v$, counted with their degrees, see \cite[Section~3]{BBC} for more details. In \cite[Section~3]{BBC} it is also shown that there exists a minimizer for
    \begin{equation}
        \label{minRelaxEnergy}
        \inf\{F_v(u):u \in H^1(\Om;\mb S^2),\; u-v \in H^1_0(\Om;\mb R^3)\},
    \end{equation}
    and every such minimizer is weakly harmonic. For maps $u$ with the same trace and the same distributional topological singularities as $v$, one has $L(u,v)=0$ and hence $F_v(u)=E_\Om(u)$. Thus, \eqref{minRelaxEnergy} can be seen as minimizing the Dirichlet energy among maps with fixed trace and prescribed singularities, through the singular set of $v$. The difficulty one faces now is that a minimizer of the relaxed problem \eqref{minRelaxEnergy} need not retain these singularities. To overcome this issue, the assumption that the trace is nonconstant is essential for the constructions in \cite{HLP,Rline,ever}. The key comparison in \cite[Lemma~7.1]{HLP} and \cite[Lemma~A1]{ever} allows one to insert a small dipole into a map with nonvanishing gradient at an energy increase strictly smaller than $8\pi$ times the length of the dipole. This strict inequality is crucial for obtaining singular minimizers of the relaxed energy. Hence, such a minimizer is a weakly harmonic map with nonempty singular set.\\
    For constant boundary data, however, the result of Brezis, Coron and Lieb \cite[Example~4 and Theorem~1.1]{BCL} shows that such a construction is in general not possible: On a ball $B \subset \mb R^3$, the infimum of the Dirichlet energy among maps with constant trace and exactly two prescribed point singularities of degrees $+1$ and $-1$ at $P$ and $N$, respectively, is
    \[
    \inf E_B(u)=8\pi|P-N|.
    \]
    This infimum is not attained: A minimizing sequence converges weakly to the constant map, while its energy concentrates along the segment joining the two singularities; see \cite[Section~VI]{BCL}. Thus, the prescribed singularities are not preserved under weak $H^1$ convergence. The same can be observed in the axially symmetric class.\\
    In view of these effects, it is natural to ask whether there exists a weakly harmonic map on a ball with constant trace and nonempty singular set. Since every regular weakly harmonic map on a ball with constant trace is constant (see for example \cite[Theorem~1]{ChouZhu1995}), we can ask equivalently:
    \[
    \textit{Do there exist nonconstant weakly harmonic maps on a ball with constant trace?}
    \]
    This question is posed explicitly in \cite[p.~198]{ever}. In this paper we provide a positive answer by constructing a nonconstant weakly harmonic map with constant trace.\\
    Before we state the main result precisely, let us give a brief overview over some connected results. Related work on constant boundary data includes the constancy results for smooth harmonic maps of Karcher and Wood~\cite{KarcherWood1984}
    and those for smooth harmonic maps and critical points for nematic liquid-crystals of Chou and Zhu~\cite{ChouZhu1995}. For a class of quasilinear elliptic systems, Shen and
    Yan~\cite{ShenYan1993} constructed nonconstant solutions with constant
    boundary values. \\
    Using a higher-dimensional analogue of the relaxed energy, Pakzad \cite{Pakzad2001} obtained infinitely many weakly harmonic extensions into $\mb S^2$ for nonconstant smooth boundary data in dimensions $n\geq3$. A complementary approach to relaxation for maps into spheres was
    developed by Giaquinta, Modica and Sou\v{c}ek~\cite{GMS1989} using
    the theory of Cartesian currents. Regularity questions for related relaxed
    variational problems were studied by Bethuel and
    Brezis~\cite{BethuelBrezis1991}.
    
    We denote $B_r:=B_r(0)\subset \mb R^3$ and by $(e_1,e_2,e_3)$ the standard basis of $\mb R^3$. The main result of this paper is the following theorem.
    \begin{theorem}
        \label{thm:main}
        There exists a weakly harmonic map $u \in H^1(B_2;\mb S^2)$ with $u-e_3\in H^1_0(B_2;\mb R^3)$ and $E_{B_2}(u)>0$. In particular, $u$ is nonconstant with constant trace on $\p B_2$.
    \end{theorem}
    Our proof uses exactly the idea of minimizing the Dirichlet energy over an admissible class with prescribed singularities and constant trace. So the central question is of course how we can overcome this concentration in minimizing sequences as observed in \cite{BCL}. To explain this we first define the admissible class of maps $\mc C$.\\ Denote the domain symmetry
    \[
    g:\mb R^3 \to \mb R^3,\quad g(x,y,z)=(-x,-y,z)
    \]
    and, for $\alpha\in\mb R$, the target rotation $R_\alpha\in\SO(3)$ through angle $\alpha$ about the $e_3$-axis,
    \[
    R_\alpha(v_1,v_2,v_3)
    =(v_1\cos\alpha-v_2\sin\alpha,\,
      v_1\sin\alpha+v_2\cos\alpha,\,v_3).
    \]
    We write for simplicity $R=R_\pi$, so $R(v_1,v_2,v_3)=(-v_1,-v_2,v_3)$.
    Let $\Om\subset \mb R^3$ be open. For $v \in H^1_\text{loc}(\Om;\mb S^2)$ we define the vector field
    \begin{equation}
        \label{eq:DefJacVec}
        D(v)=(v\cdot (v_y\times v_z),v\cdot (v_z\times v_x),v\cdot (v_x\times v_y)) \in L^1_{loc}(\Om;\mb R^3),
    \end{equation}
    as introduced in \cite{BCL}. This is well-defined almost everywhere and we have
    \begin{equation}
        \label{ineq:UpperBoundJac}
        |D(v)|\leq \frac{1}{2}|\nabla v|^2,
    \end{equation}
    see \cite[equation~(4.1)]{BCL}.\\
    From now on fix once and for all some $p \in (2,3)$. We define $\mc C$ to be the class of maps $v:\mb R^3\to \mb S^2$ such that
    \begin{align}
        &v-e_3 \in W^{1,p}(\mb R^3;\mb R^3),\quad v=e_3\text{ a.e. on }\mb R^3\setminus B_2\nonumber \\
        &v\circ g=R\circ v,\text{ a.e. on }\mb R^3\\
        &\di D(v)=4\pi(\delta_{e_1}+\delta_{-e_1}-\delta_{e_3}-\delta_{-e_3})\text{ in }\mc D'(\mb R^3).\nonumber
    \end{align}
    \begin{remark}
        The additional regularity of $v-e_3$ beyond $H^1$ is not strictly necessary for the proof, but it makes some of the later arguments simpler due to the existence of continuous slices. The equation for $\di D(v)$ gives a concise way of prescribing topological charges for $v \in \mc C$.
    \end{remark}
    With this we can outline the strategy behind the proof. The key observation is that the symmetry condition and its interplay with the prescribed charges in $\mc C$ make the collapse of sequences in $\mc C$ to a constant energetically very expensive. This is shown in Proposition \ref{prop:Constant-expensive}. In fact, the precise cost of such a collapse is bounded below by $32\pi$. Now observe the following: For maps $v \in \mc C$ the minimal connection has length $2\sqrt{2}$, as the distance from either positive singularity to either negative singularity is $\sqrt{2}$. The lower bound of Brezis, Coron and Lieb \cite[Sections~II and~IV]{BCL} therefore gives
    \[
    E_{B_2}(v)\geq 8\pi\,(2\sqrt{2})=16\sqrt{2}\pi.
    \]
    Hence, there is hope to actually find a map $v \in \mc C$, which has
    \[
    E_{B_2}(v)<32\pi,
    \]
    and thus prevents a minimizing sequence in $\mc C$ from collapsing to a constant. This competitor is constructed in Section \ref{sec:Competitor}. In Section \ref{sec:Main} we complete the proof as follows: First, by combining the observations of Section \ref{sec:LowerBound} and \ref{sec:Competitor} it is not difficult to see that a minimizing sequence for the Dirichlet energy in $\mc C$ converges (up to a subsequence) weakly in $H^1$ to a nonconstant map $u$ with constant trace. Finally, a first-variation argument involving local rotations that preserve $\mc C$ establishes weak harmonicity of $u$.

    \section{Lower bound}
    \label{sec:LowerBound}
    For a continuous map $w:\mb R^2\to\mb S^2$ equal to $e_3$ outside a compact set, we denote by $\deg w$ the degree of its continuous extension to the one-point compactification $\mb R^2\cup\{\infty\}\cong\mb S^2$, with $w(\infty)=e_3$. For a smooth map $f:\mb R^2\to\mb S^2$ constant outside a compact set, the degree formula is
    \begin{equation}
        \label{formulaDegree}
        \deg f
        =\frac1{4\pi}\int_{\mb R^2}
        f\cdot(\partial_{\xi_1}f\times\partial_{\xi_2}f)\d\xi
        =\frac1{4\pi}\int_0^{2\pi}\int_0^\infty
        f\cdot(f_r\times f_\theta)\d r\d\theta,
    \end{equation}
    where $\xi=(r\cos\theta,r\sin\theta)$.
    \begin{lemma}
        \label{lem:OddDegree}
        Suppose $w:\mb R^2 \to \mb S^2$ is continuous and equals $e_3$ outside a compact set. Moreover, suppose $w$ satisfies
        \begin{equation}
            \label{eq:Symmetryw}
            w(-\xi)=Rw(\xi),\quad \forall \xi \in \mb R^2.
        \end{equation}
        If $\deg w$ is odd, then $w(0)=-e_3$.
    \end{lemma}

    \begin{proof}
    The symmetry assumption \eqref{eq:Symmetryw} gives $w(0)=Rw(0)$, so $w(0)\in\{e_3,-e_3\}$.
    We prove the contrapositive: Assume $w(0)=e_3$ and show that
    $\deg w$ is even.\\
    First, we reduce to the case in which $w$ is smooth and identically
    $e_3$ near the origin. By continuity, there is $r_0>0$ such that
    $w(\xi)\cdot e_3>0$ whenever $|\xi|<r_0$. Choose a smooth radial
    cutoff $\chi:\R^2\to[0,1]$, supported in $\{|\xi|<r_0\}$ and equal
    to one near zero. For $t\in[0,1]$, define the homotopy
    \[
    H_t(\xi)=
    \frac{(1-t\chi(\xi))w(\xi)+t\chi(\xi)e_3}
    {\bigl|(1-t\chi(\xi))w(\xi)+t\chi(\xi)e_3\bigr|}.
    \]
    This is well-defined: wherever $\chi(\xi)\ne0$, the
    numerator has positive third component, and wherever $\chi(\xi)=0$,
    it equals $w(\xi)$. Since $\chi$ is radial and $Re_3=e_3$, we have
    $H_t(-\xi)=RH_t(\xi)$. Moreover, $H_t=w$ outside the support of $\chi$,
    so the homotopy extends continuously across infinity. Thus $H_1$ has
    the same degree as $w$, retains the symmetry and the constant value
    outside a compact set, and equals $e_3$ near zero. Hence by replacing $w$ with $H_1$ (and calling it $w$ again), we can suppose that $w$ equals $e_3$ in a neighborhood of the origin.\\
    Next we want to show that we can smooth while preserving the degree and symmetry. For $\eps>0$, let $\rho_\eps$ be a nonnegative smooth radial mollifier of integral
    one, supported in $\{|\xi|<\eps\}$. By uniform continuity of $w$, we have $\rho_\eps * w\to w$ uniformly as
    $\eps\downarrow0$. In particular, for $\eps>0$ small this means that
    \[
    w_\eps:=\frac{\rho_\eps * w}{|\rho_\eps * w|}
    \]
    is a smooth sphere-valued map. Radiality of the mollifier implies $w_\eps(-\xi)=Rw_\eps(\xi)$. Taking $\eps$ small enough ensures that $w_\eps=e_3$ near zero; compact support of the
    mollifier ensures that $w_\eps=e_3$ outside a sufficiently large disk.
    As before it is easy to see that $w$ and $w_\eps$ are homotopic for $\eps$ sufficiently small. Consequently $\deg w_\eps=\deg w$. We now fix such an $\eps$
    and again denote $w_\eps$ by $w$.\\
    In polar coordinates
    $\xi=(r\cos\theta,r\sin\theta)$, define
    \[
    u(r,\theta)=R_{-\theta}w(r,\theta),\qquad r>0.
    \]
    Because $w=e_3$ near zero and infinity, the same holds for $u$ and, in particular, both maps extend smoothly to $\mb R^2 \cup \{\infty\}$.
    Assumption \eqref{eq:Symmetryw} gives
    \[
    u(r,\theta+\pi)
    =R_{-\theta-\pi}Rw(r,\theta)
    =R_{-\theta}w(r,\theta)=u(r,\theta).
    \]
    Thus $u=\bar u\circ q$ for a continuous map $\bar u:\mb S^2\to \mb S^2$,
    where $q:\mb R^2 \cup \{\infty\}\to \mb R^2 \cup \{\infty\}$ is continuous with
    \[
    q(r,\theta)=(r,2\theta),\quad q(0)=0,\quad q(\infty)=\infty.
    \]
    In particular, $\deg q=2$,
    and multiplicativity of degree yields that $\deg u$ is even.\\
    It remains to compare $\deg u$ with $\deg w$ using \eqref{formulaDegree}.
    Since $\frac{d}{d\alpha}(R_\alpha a)
    =R_\alpha(e_3\times a)$ for a fixed vector $a$, differentiation gives
    \[
    u_r=R_{-\theta}w_r,\qquad
    u_\theta=R_{-\theta}(w_\theta-e_3\times w).
    \]
    Rotations preserve scalar triple products. Also, $|w|=1$ implies
    $w_r\cdot w=0$, so the vector triple-product identity gives
    \[
    w_r\times(e_3\times w)
    =e_3(w_r\cdot w)-w(w_r\cdot e_3)
    =-(\partial_r w_3)w.
    \]
    It follows that
    \[
    u\cdot(u_r\times u_\theta)
    =w\cdot\bigl(w_r\times(w_\theta-e_3\times w)\bigr)
    =w\cdot(w_r\times w_\theta)+\partial_r w_3.
    \]
    Applying the degree formula, we obtain
    \[
    \deg u-\deg w
    =\frac1{4\pi}\int_0^{2\pi}\int_0^\infty
    \partial_r w_3d rd\theta
    =\frac1{4\pi}\int_0^{2\pi}
    \bigl(w_3(\infty,\theta)-w_3(0,\theta)\bigr)d\theta=0,
    \]
    because $w=e_3$ near both zero and infinity. Therefore, $\deg w=\deg u$ is even, which completes the proof.
    \end{proof}

    \begin{lemma}
    \label{lem:disk-energy}
    For $\rho>0$, denote $B_\rho^2:=\{\xi\in\R^2:|\xi|<\rho\}$. Let
    $w\in H^1(B_\rho^2;\mb S^2)$ be continuous with $w(0)=-e_3$ and satisfying $w(-\xi)=Rw(\xi)$ for each $\xi \in B^2_\rho$.\\
    Then
    \begin{equation}\label{eq:disk-energy}
    E_{B_\rho^2}(w)
    \ge8\pi-\frac2{\rho^2}\int_{B_\rho^2}|w-e_3|^2 d\xi.
    \end{equation}
    \end{lemma}
    \begin{proof}
    We work in polar coordinates $\xi =(r\cos \theta,r\sin\theta)$. By the symmetry assumption on $w$, we have for $i=1,2$ that
    \[
    w_i(r,\theta+\pi)=-w_i(r,\theta).
    \]
    In particular, this implies
    \[
    \frac{1}{2\pi}\int_{0}^{2\pi}w_i(r,\theta)d\theta =0,\quad \text{for }i=1,2.
    \]
    Hence, Poincar\'e inequality gives, for almost every $r$,
    \[
    \int_0^{2\pi}|w_\theta|^2 d\theta
    \ge\int_0^{2\pi}(w_1^2+w_2^2) d\theta
    =\int_0^{2\pi}(1-w_3^2) d\theta.
    \]
    Since $|w|=1$, we have
    $w_r\cdot(e_3-w_3w)=\partial_r w_3$ and $|e_3-w_3w|^2=1-w_3^2$.
    Completing the square gives
    \[
    r|w_r|^2+\frac{1-w_3^2}{r}\ge2\partial_r w_3.
    \]
    Combining the above inequalities and integrating over $\eps<r<s<\rho$ yields
    \[
    E_{B_s^2\setminus B_\eps^2}(w)=\int_{\eps}^s\int_{0}^{2\pi}\left( r|w_r|^2+\frac{|w_\theta|^2}{r}\right)d\theta dr
    \ge2\int_0^{2\pi}
    \bigl(w_3(s,\theta)-w_3(\eps,\theta)\bigr)d\theta.
    \]
    Sobolev slicing justifies this calculation for almost every pair of
    radii away from zero. Continuity at $0$ and $w_3(0)=-1$ allow
    $\eps\downarrow0$, giving, for almost every $s<\rho$,
    \[
    E_{B_s^2}(w)
    \ge2\int_0^{2\pi}(1+w_3(s,\theta)) d\theta
    =8\pi-\int_0^{2\pi}|w(s,\theta)-e_3|^2 d\theta,
    \]
    where we used $|w-e_3|^2=2-2w_3$. Since $E_{B_\rho^2}(w)\geq E_{B_s^2}(w)$ this means
    \[
    E_{B_\rho^2}(w) \geq 8\pi-\int_0^{2\pi}|w(s,\theta)-e_3|^2 d\theta.
    \]
    Multiply the above by $s$ and integrate from $0$ to $\rho$. The left-hand side
    is independent of $s$, so
    \[
    \frac{\rho^2}{2}E_{B_\rho^2}(w)
    \ge 4\pi\rho^2
    -\int_0^\rho\int_0^{2\pi}
    |w(s,\theta)-e_3|^2\,s\,d\theta\,ds
    =4\pi\rho^2-\int_{B_\rho^2}|w(\xi)-e_3|^2\,d\xi,
    \]
    which after rearranging is the desired estimate \eqref{eq:disk-energy}.
    \end{proof}

    \begin{proposition}
        \label{prop:Constant-expensive}
        Suppose $(v_n)_{n\in \mb N} \subset \mc C$ with $v_n \to e_3$ in $L^2(B_2)$. Then
        \[
        \liminf_{n \to \infty} E_{B_2}(v_n) \geq 32\pi.
        \]
    \end{proposition}

    \begin{proof}
        Let $v \in \mc C$. Recall the definition of the vector field $D(v)$ in \eqref{eq:DefJacVec}. We define the functions
        \[
        F_z(t)=\int_{\mb R^2}D_3(v)(x,y,t)dxdy,\quad F_x(t)=\int_{\mb R^2} D_1(v)(t,y,z)dydz.
        \]
        We want to compute the values of these functions precisely. To do so, recall that since $v \in \mc C$, we have
        \begin{equation}
            \label{eq:DivEq}
            \di D(v)=4\pi (\delta_{e_1}+\delta_{-e_1}-\delta_{e_3}-\delta_{-e_3}),\quad \text{in }\mc D'(\mb R^3).
        \end{equation}
        Now pick $\eta \in C_c^\infty(\mb R)$ and let $\chi \in C_c^\infty(\mb R^2)$ such that $\chi\equiv 1$ on $\{(x,y):\sqrt{x^2+y^2}\leq 2\}$. Hence, 
        \[
        \varphi(x,y,z)=\chi(x,y)\eta(z)
        \]
        is an admissible test function for \eqref{eq:DivEq}. Since $D(v)=0$ outside $B_2$ as $v$ is constant there, we compute
        \begin{align*}
            \langle \di D(v),\varphi \rangle&=-\int_{\mb R}\left(\int_{\mb R^2} D_3(v)(x,y,z)dxdy\right)\eta'(z)dz \\
            &=-\int_{\mb R}F_z(z)\eta'(z)dz\\
            &=\langle F_z',\eta \rangle.
        \end{align*}
        From \eqref{eq:DivEq} we have hence
        \[
        \langle F_z',\eta\rangle=4\pi (\varphi(e_1)+\varphi(-e_1)-\varphi(e_3)-\varphi(-e_3))=4\pi(2\eta(0)-\eta(1)-\eta(-1)),
        \]
        in other words
        \[
        F_z'=4\pi(2\delta_0-\delta_{-1}-\delta_1).
        \]
        Since $F_z$ has compact support, it follows that
        \begin{equation}
            \label{eq:FormFz}
            F_z(t)=
        \begin{cases}
        -4\pi,&-1<t<0,\\
        4\pi,&0<t<1,\\
        0,&|t|>1,
        \end{cases}
        \quad\text{for a.e. }t.
        \end{equation}
        A similar computation shows that
        \begin{equation}
            \label{eq:FormFx}
            F_x(t)=
        \begin{cases}
        4\pi,&-1<t<0,\\
        -4\pi,&0<t<1,\\
        0,&|t|>1,
        \end{cases}
        \quad\text{for a.e. }t.
        \end{equation}
        Next, note that for a.e. $t$ the slice $v(\cdot,\cdot,t)-e_3$ belongs to $W^{1,p}(\mb R^2)$. In particular, since $p>2$ we can suppose that $v(\cdot,\cdot,t)$ is continuous. The definition of the admissible class gives further that $v(\cdot,\cdot,t)$ satisfies the symmetry condition of Lemma \ref{lem:OddDegree} and is equal to $e_3$ outside a compact set. Finally, using the degree formula \eqref{formulaDegree} together with the smoothing argument already discussed in the proof of Lemma \ref{lem:OddDegree} we have by \eqref{eq:FormFz}
        \begin{equation}
            \label{eq:degree}
            |\deg(v(\cdot,\cdot,t))|=\frac{|F_z(t)|}{4\pi}=1,\quad \text{for a.e. }-1<t<1.
        \end{equation}
        Lemma \ref{lem:OddDegree} gives that
        \begin{equation*}
            v(0,0,t)=-e_3,\quad \text{for a.e. }-1<t<1.
        \end{equation*}
        Fix $0<\rho<1$. Apply Lemma~\ref{lem:disk-energy} to $v(\cdot,\cdot,t)$ and integrate over $-1<t<1$. This gives
        \begin{equation}\label{eq:cylinder-energy}
        \int_{\{x^2+y^2<\rho^2,|z|<1\}}(|\partial_xv|^2+|\partial_yv|^2)d x d y d z
        \geq 16\pi-\frac2{\rho^2}||v-e_3||_{L^2(B_2)}^2.
        \end{equation}
        On the other hand, $2|D_1(v)|\le|\partial_yv|^2+|\partial_zv|^2$
        and \eqref{eq:FormFx} give
        \begin{equation}\label{eq:side-energy}
        \int_{\{|x|>\rho\}}
        (|\partial_yv|^2+|\partial_zv|^2)d x d y d z
        \ge2\int_{|s|>\rho}|F_x(s)|d s=16\pi(1-\rho).
        \end{equation}
        Since the two regions of integration are disjoint, adding
        \eqref{eq:cylinder-energy} and \eqref{eq:side-energy} yields
        \begin{equation}\label{eq:quantitative-threshold}
        E_{B_2}(v)\ge32\pi-16\pi\rho
        -\frac2{\rho^2}||v-e_3||_{L^2(B_2)}^2.
        \end{equation}
        Apply \eqref{eq:quantitative-threshold} to $v_n$, first let $n\to\infty$, and then let
        $\rho\downarrow0$. This completes the proof.
    \end{proof}

    \section{Competitor}

    \label{sec:Competitor}

    \begin{proposition}
        \label{prop:trial}
        The class $\mc C$ is nonempty. More precisely, define $Q:\mb S^2\to \mb S^2$ by
        \[
        Q(\xi)=
        \begin{cases}
        (-2\xi_1\xi_3,-2\xi_2\xi_3,1-2\xi_3^2),&\xi_3<0,\\
        e_3,&\xi_3\ge0.
        \end{cases}
        \]
        Let $\eps>0$ and denote $r=\sqrt{x^2+y^2+z^2}$. Define
        \begin{equation}
        \label{eq:trial}
        \begin{gathered}
        F_\eps(x,y,z)=(xz,-r y,\eps(r^2-1)),\quad
        v_\eps=Q\left(\frac{F_\eps}{|F_\eps|}\right)
        \quad\text{on }\R^3\setminus\{\pm e_1,\pm e_3\}.
        \end{gathered}
        \end{equation}
        Then $v_\eps\in\mc C$. Moreover, $v_\eps=e_3$ almost everywhere
        outside $B_1$.
    \end{proposition}

    \begin{proof}
    First note that, since $F_\eps\circ g=RF_\eps$ and
    $Q\circ R=R\circ Q$, the map $v_\eps$ has the required symmetry.
    Next, outside $B_1$, the map $F_\eps/|F_\eps|$ takes values in the
    upper hemisphere wherever it is defined. Hence $v_\eps=e_3$ almost
    everywhere outside $B_1$.\\
    It is easy to see that $F_\eps\in C^1(\R^3;\R^3)$ and that its only
    zeros are $\{\pm e_1,\pm e_3\}$. Since $Q$ is Lipschitz, it is clear
    that $v_\eps$ is locally Lipschitz away from these four points.
    Suppose $a\in\{\pm e_1,\pm e_3\}$. An elementary computation shows
    that
    \[
    \det DF_\eps(a)=2\eps(a_1^2-a_3^2)\ne0.
    \]
    In particular, by the inverse function theorem, we have
    \[
    |F_\eps(x,y,z)|\ge c_\eps|(x,y,z)-a|
    \]
    near $a$. Again, from the Lipschitz continuity of $Q$ and the definition
    of $v_\eps$, it follows that
    \[
    |\nabla v_\eps(x,y,z)|\le C_\eps|(x,y,z)-a|^{-1}
    \]
    almost everywhere near $a$. These estimates hold near each of the four
    points, and hence $\nabla v_\eps\in L^s(\R^3)$ for every $1\le s<3$.
    Since $v_\eps$ is bounded, a standard cutoff argument shows that its
    weak derivatives extend across the set $\{\pm e_1,\pm e_3\}$. Because $v_\eps-e_3$
    has compact support, this proves that
    \[
    v_\eps-e_3\in W^{1,s}(\R^3;\R^3)
    \qquad\text{for every }1\le s<3.
    \] 
    Next we prove the formula for the distributional divergence,
    \begin{equation}\label{eq:trial-J}
    \ddiv D(v_\eps)
    =4\pi(\delta_{e_1}+\delta_{-e_1}-\delta_{e_3}-\delta_{-e_3})
    \quad\text{in }\mc D'(\R^3).
    \end{equation}
    In view of \cite[Appendix~B]{BCL},
    it is enough to show that $e_1,-e_1$ are isolated singular points of
    degree $1$, and that $e_3,-e_3$ are isolated singular points of degree
    $-1$.\\
    First note that $Q$ has degree $1$. Indeed, the unique preimage of
    $-e_3$ is $-e_3$, and $dQ(-e_3)=2I$ on the tangent plane.
    Let $a\in\{\pm e_1,\pm e_3\}$ and choose $\rho>0$ sufficiently small.
    Orient $\partial B_\rho(a)$ as the boundary of its ball. By
    multiplicativity of degree, we have
    \[
    \deg\bigl(v_\eps|_{\partial B_\rho(a)}\bigr)
    =\deg(Q)\,\operatorname{sgn}\det DF_\eps(a)
    =a_1^2-a_3^2.
    \]
    This proves that the degrees at $\pm e_1$ are $1$ and those at
    $\pm e_3$ are $-1$. Hence $v_\eps\in\mc C$, completing the proof.
    \end{proof}

    \begin{remark}
    Let us give some intuition behind the competitor. The geometric idea is to
    concentrate the energy near the cross
    \[
    [-e_1,e_1]\cup[-e_3,e_3] \subset \mb R^3.
    \]
    Recall from the proof of Proposition~\ref{prop:Constant-expensive}
    that the horizontal slices of an admissible map have odd degree for
    almost every $-1<t<1$. Lemma~\ref{lem:OddDegree} therefore gives
    $v(0,0,t)=-e_3$. For sequences converging to $e_3$,
    Lemma~\ref{lem:disk-energy} then gives a concentration cost of at least
    $16\pi$ along the vertical segment $[-e_3,e_3]$.
    To connect the remaining singularities to this segment, a natural
    symmetric choice is to join both $e_1$ and $-e_1$ to the origin. This suggests a concentration pattern with energy $32\pi$.\\
    Inside $B_1$, the stereographic coordinate of $v_\eps$ from $e_3$ is
    $(xz,-ry)/(\eps(1-r^2))$. Note that
    \[
    x^2z^2+r^2y^2=(x^2+y^2)(z^2+y^2)
    \]
    is the product of the squared distances to the two axes.
    Thus, for small $\eps$, the map is close to $e_3$ away from the cross.
    Near the origin, the four arms meet in a regular region of size
    $\sqrt\eps$. The computations below show that this is exactly the region where energy can be saved. Since $v_\eps \to e_3$ in $L^2(B_2)$ we know from
    Proposition~\ref{prop:Constant-expensive}, that 
    \[
    \liminf_{\eps \downarrow 0}E_{B_2}(v_\eps)\geq 32\pi,
    \]
    while the computations in Proposition \ref{prop:gap} show that for small $\eps>0$ we have
    \[
    E_{B_2}(v_\eps)<32\pi.
    \]
    \end{remark}

    \begin{proposition}
        \label{prop:gap}
        In the notation of Proposition~\ref{prop:trial}, we have for all
        $\eps>0$ sufficiently small that
        \begin{equation}
            \label{eq:strict-gap}
            E_{B_2}(v_\eps)<32\pi.
        \end{equation}
    \end{proposition}

    \begin{proof}
    As $v_\eps$ is almost everywhere constant outside $B_1$, we have
    $E_{B_2}(v_\eps)=E_{B_1}(v_\eps)$. Let
    $\omega=(\omega_1,\omega_2,\omega_3)\in\mb S^2$ and denote, for
    $0<r<1$,
    \[
    k=k(r)=\frac{\eps(1-r^2)}{r^2}.
    \]
    From the definition of $v_\eps$ in \eqref{eq:trial}, we have in polar
    coordinates
    \[
    v_\eps(r\omega)
    =\frac{(2k\omega_1\omega_3,-2k\omega_2,
    \omega_1^2\omega_3^2+\omega_2^2-k^2)}
    {\omega_1^2\omega_3^2+\omega_2^2+k^2}.
    \]
    One computes first
    \begin{equation}
        \label{eq:radial-energy}
        |\partial_r v_\eps|^2
        =\frac{4k'^2(\omega_1^2\omega_3^2+\omega_2^2)}
        {(\omega_1^2\omega_3^2+\omega_2^2+k^2)^2}.
    \end{equation}
    For the spherical gradient we have
    \begin{equation}
        \label{eq:angular-energy}
        |\nabla_{\mb S^2}v_\eps|^2
        =\frac{4k^2}{(\omega_1^2\omega_3^2+\omega_2^2+k^2)^2}
        \left(|\nabla_{\mb S^2}(\omega_1\omega_3)|^2
        +|\nabla_{\mb S^2}\omega_2|^2\right).
    \end{equation}
    Now
    \begin{equation}
        \label{eq:spherical-gradients}
        |\nabla_{\mb S^2}\omega_2|^2=1-\omega_2^2,\qquad
        |\nabla_{\mb S^2}(\omega_1\omega_3)|^2
        =1-\omega_2^2-4\omega_1^2\omega_3^2.
    \end{equation}
    Combining \eqref{eq:radial-energy}--\eqref{eq:spherical-gradients},
    it follows
    \begin{equation}
        \label{eq:raw-energy}
        r^2|\partial_r v_\eps|^2+|\nabla_{\mb S^2}v_\eps|^2
        =\frac{
        4(r^2k'^2-4k^2)(\omega_1^2\omega_3^2+\omega_2^2)
        +8k^2(1+\omega_2^2)}
        {(\omega_1^2\omega_3^2+\omega_2^2+k^2)^2}.
    \end{equation}
    Next we rewrite this formula using terms which simplify after
    integrating. First, one can compute
    \begin{equation}
        \label{eq:angular-ibp}
        \Delta_{\mb S^2}\log(\omega_1^2\omega_3^2+\omega_2^2+k^2)
        =\frac{4k^2(1-3\omega_1^2\omega_3^2+k^2)}
        {(\omega_1^2\omega_3^2+\omega_2^2+k^2)^2}-4
    \end{equation}
    and
    \begin{equation}
        \label{eq:radial-derivative}
        \begin{aligned}
        8\frac{d}{dr}\left(
        \frac{rk^2}{\omega_1^2\omega_3^2+\omega_2^2+k^2}\right)
        =\frac{8k^2}{\omega_1^2\omega_3^2+\omega_2^2+k^2}
        +\frac{16rkk'(\omega_1^2\omega_3^2+\omega_2^2)}
        {(\omega_1^2\omega_3^2+\omega_2^2+k^2)^2}.
        \end{aligned}
    \end{equation}
    Finally, note that $rk'+2k=-2\eps$, which implies
    \begin{equation}
        \label{eq:k-identity}
        16\eps^2-16rkk'=4r^2k'^2+16k^2.
    \end{equation}
    Starting from \eqref{eq:raw-energy}, we compute

\begin{equation}
\label{eq:energy-decomposition}
\begin{aligned}
&r^2|\partial_r v_\varepsilon|^2
   +|\nabla_{\mb S^2}v_\varepsilon|^2 \\
&=
\frac{
4(r^2k'^2-4k^2)(\omega_1^2\omega_3^2+\omega_2^2)
+8k^2(1+\omega_2^2)
}{
(\omega_1^2\omega_3^2+\omega_2^2+k^2)^2
} \\
&=
\frac{
8k^2(1-3\omega_1^2\omega_3^2+k^2)
}{
(\omega_1^2\omega_3^2+\omega_2^2+k^2)^2
} \\
&\quad
-\frac{
8k^2(\omega_1^2\omega_3^2+\omega_2^2+k^2)
+16rkk'(\omega_1^2\omega_3^2+\omega_2^2)
}{
(\omega_1^2\omega_3^2+\omega_2^2+k^2)^2
} \\
&\quad
+16\frac{
\varepsilon^2(\omega_1^2\omega_3^2+\omega_2^2)
-k^2\omega_2^2
}{
(\omega_1^2\omega_3^2+\omega_2^2+k^2)^2
} \\
&=
8
+2\Delta_{\mb S^2}
\log(\omega_1^2\omega_3^2+\omega_2^2+k^2) \\
&\quad
-8\frac{d}{dr}
\left(
\frac{rk^2}
{\omega_1^2\omega_3^2+\omega_2^2+k^2}
\right) \\
&\quad
+16\frac{
\varepsilon^2(\omega_1^2\omega_3^2+\omega_2^2)
-k^2\omega_2^2
}{
(\omega_1^2\omega_3^2+\omega_2^2+k^2)^2
},
\end{aligned}
\end{equation}
    where we used \eqref{eq:k-identity} in the second and \eqref{eq:angular-ibp}--\eqref{eq:radial-derivative} in the last equality.\\
    For each $0<r<1$, the function
    $\log(\omega_1^2\omega_3^2+\omega_2^2+k^2)$ is smooth on $\mb S^2$.
    Hence, by integration by parts,
    \begin{equation}
        \label{eq:angular-zero}
        \int_{\mb S^2}\Delta_{\mb S^2}
        \log(\omega_1^2\omega_3^2+\omega_2^2+k^2)d\omega=0.
    \end{equation}
    For the radial derivative term in \eqref{eq:energy-decomposition}, integrating first over
    $a<r<b$, where $0<a<b<1$, gives only the boundary contributions
    \[
    -8\left[\int_{\mb S^2}
    \frac{rk^2}{\omega_1^2\omega_3^2+\omega_2^2+k^2}d\omega
    \right]_{r=a}^{r=b}.
    \]
    As $r\downarrow0$, we have
    \[
    0\le\int_{\mb S^2}
    \frac{rk^2}{\omega_1^2\omega_3^2+\omega_2^2+k^2}d\omega
    \le4\pi r\longrightarrow0.
    \]
    As $r\uparrow1$, we have $k(r)\to0$, and the integrand above
    converges to zero almost everywhere and is bounded by $1$.
    Indeed, $\omega_1^2\omega_3^2+\omega_2^2$ vanishes only at
    $\{\pm e_1,\pm e_3\}$. Thus dominated convergence shows that
    \[
    \int_{\mb S^2}
    \frac{rk^2}{\omega_1^2\omega_3^2+\omega_2^2+k^2}\d\omega
    \longrightarrow0\quad\text{as }r\uparrow1.
    \]
    In other words we have
    \begin{equation}
        \label{eq:NoradContr}
        \int_{0}^1\int_{\mb S^2}\frac{d}{dr}\left(\frac{rk^2}{\omega_1^2\omega_3^2+\omega_2^2+k^2}\right)d\omega dr=0.
    \end{equation}
    Before passing to the radial endpoints in
    \eqref{eq:energy-decomposition}, we have to check that the remaining
    term is actually integrable. For its positive contribution, the substitution
    $t=k(r)$ and the inequality $-k'(r)=2\eps/r^3\ge2\eps$ give,
    by Tonelli's theorem,
    \begin{equation}
        \label{eq:positive-correction}
        \begin{aligned}
        &16\eps^2\int_0^1\int_{\mb S^2}
        \frac{\omega_1^2\omega_3^2+\omega_2^2}
        {(\omega_1^2\omega_3^2+\omega_2^2+k^2)^2}\d\omega\d r\\
        &\qquad\le8\eps\int_{\mb S^2}\int_0^\infty
        \frac{\omega_1^2\omega_3^2+\omega_2^2}
        {(\omega_1^2\omega_3^2+\omega_2^2+t^2)^2}\d t\d\omega\\
        &\qquad=2\pi\eps\int_{\mb S^2}
        \frac{1}{\sqrt{\omega_1^2\omega_3^2+\omega_2^2}}\d\omega.
        \end{aligned}
    \end{equation}
    The last integral is finite and independent of $\eps$. Indeed,
    \[
    \omega_1^2\omega_3^2+\omega_2^2
    =(1-\omega_1^2)(1-\omega_3^2),
    \]
    so, near each of $\{\pm e_1,\pm e_3\}$, its reciprocal square root
    is bounded by a constant times the reciprocal of the distance to
    that point, which is integrable on $\mb S^2$.
    The negative contribution in \eqref{eq:energy-decomposition} is also integrable, since
    \[
    0\le
    \frac{k^2\omega_2^2}
    {(\omega_1^2\omega_3^2+\omega_2^2+k^2)^2}
    \le\frac14.
    \]
    We can therefore integrate \eqref{eq:energy-decomposition} over
    $\mb S^2\times(a,b)$ and let $a\downarrow0$ and $b\uparrow1$.
    Using \eqref{eq:angular-zero}--\eqref{eq:NoradContr}, we obtain
    \begin{equation}
        \label{eq:exact-energy}
        E_{B_2}(v_\eps)=E_{B_1}(v_\eps)=32\pi+16\int_0^1\int_{\mb S^2}
        \frac{\eps^2(\omega_1^2\omega_3^2+\omega_2^2)-k^2\omega_2^2}
        {(\omega_1^2\omega_3^2+\omega_2^2+k^2)^2}\d\omega\d r.
    \end{equation}
    It remains to estimate the negative contribution in the integral above. For
    $0<\eps\le1/8$ and $\sqrt\eps\le r\le2\sqrt\eps$, we have
    \[
    \frac18\le\frac14-\eps\le k(r)\le1-\eps\le1.
    \]
    Since $\omega_1^2\omega_3^2+\omega_2^2\le1$, it follows that
    \begin{equation}
        \label{eq:negative-gap}
        \begin{aligned}
        16\int_0^1\int_{\mb S^2}
        \frac{k^2\omega_2^2}
        {(\omega_1^2\omega_3^2+\omega_2^2+k^2)^2}\d\omega\d r
        \ge\frac1{16}\int_{\sqrt\eps}^{2\sqrt\eps}
        \int_{\mb S^2}\omega_2^2\d\omega\d r
        =\frac\pi{12}\sqrt\eps.
        \end{aligned}
    \end{equation}
    Combining \eqref{eq:positive-correction}, \eqref{eq:exact-energy}
    and \eqref{eq:negative-gap}, we conclude that
    \[
    E_{B_2}(v_\eps)\le32\pi-\frac\pi{12}\sqrt\eps
    +2\pi\eps\int_{\mb S^2}
    \frac{1}{\sqrt{\omega_1^2\omega_3^2+\omega_2^2}}\d\omega
    <32\pi
    \]
    for all sufficiently small $\eps>0$. This completes the proof.
    \end{proof}

    \section{Proof of the main result}
    \label{sec:Main}
    \begin{lemma}
        \label{lem:RotPreserveDiv}
        Let $v \in H_{loc}^1(\mb R^3;\mb S^2)$ and $\psi \in C_c^\infty(\mb R^3;\mb R^3)$. Suppose $v_t$ solves
        \[
        \partial_t v_t =\psi \times v_t,\quad v_0=v.
        \]
        Then we have $\di D(v_t)=\di D(v)$ in $\mc D'(\mb R^3)$.
    \end{lemma}

    \begin{proof}
        The flow can be written as $v_t=Q_tv$, where $Q_t$ is a
        smooth field of rotations. Thus, $t\mapsto v_t$ is differentiable
        locally in $H^1$ and $L^\infty$, and $t\mapsto D(v_t)$ is
        differentiable in $L^1_{\mathrm{loc}}$. We claim
        \begin{equation}\label{eq:rotation-curl}
        \partial_tD(v_t)
        =\operatorname{curl}
        \bigl(\psi\cdot\partial_xv_t,\,
              \psi\cdot\partial_yv_t,\,
              \psi\cdot\partial_zv_t\bigr),\quad \text{in }\mc D'(\mb R^3).
        \end{equation}
        We verify \eqref{eq:rotation-curl} for the first component, the other components follow similarly. Differentiate
        $D_1(v_t)=v_t\cdot(\partial_yv_t\times\partial_zv_t)$, using
        \[
        \partial_t\partial_i v_t
        =\partial_i\psi\times v_t+\psi\times\partial_i v_t.
        \]
        The terms without derivatives of $\psi$ cancel by rotational
        invariance of the scalar triple product. Since $|v_t|=1$ and
        $v_t\cdot\partial_i v_t=0$, the remaining terms give
        \[
        \begin{aligned}
        \partial_tD_1(v_t)
        &=v_t\cdot\bigl(
          (\partial_y\psi\times v_t)\times\partial_zv_t
          +\partial_yv_t\times(\partial_z\psi\times v_t)
          \bigr)\\
        &=\partial_y\psi\cdot\partial_zv_t
          -\partial_z\psi\cdot\partial_yv_t\\
        &=\partial_y(\psi\cdot\partial_zv_t)
          -\partial_z(\psi\cdot\partial_yv_t).
        \end{aligned}
        \]
        The last equality holds distributionally because mixed weak
        derivatives commute. \\
        Taking the distributional divergence in
        \eqref{eq:rotation-curl} yields
        \[
        \partial_t\operatorname{div}D(v_t)=0.
        \]
        Since $v_0=v$, we conclude that
        $\operatorname{div}D(v_t)=\operatorname{div}D(v)$ for every $t$.
    \end{proof}

    \begin{proof}[Proof of Theorem \ref{thm:main}]
        Pick $\eps>0$ sufficiently small such that by Proposition \ref{prop:gap} we have
        \begin{equation}
            \label{eq:mSmall}
            m:=\inf_{v\in \mc C} E_{B_2}(v)\leq E_{B_2}(v_\eps)<32\pi.
        \end{equation}
        Let $(v_n)_{n\in \mb N}\subset \mc C$ be a minimizing sequence, that is $E_{B_2}(v_n)\to m$. Since $|v_n|=1$, the sequence is bounded in $H^1(B_2)$. Thus, after taking subsequences which we do not relabel, we have
        \begin{equation}
            \label{eq:ConvMin}
            v_n \rightharpoonup u,\quad \text{in }H^1(B_2),\quad v_n\to u,\quad \text{in }L^2(B_2)\text{ and almost everywhere.}
        \end{equation}
        Note that this implies $|u|=1$ almost everywhere and $u-e_3 \in H^1_0(B_2;\mb R^3)$. We claim that $u$ is weakly harmonic.\\
        Fix for the moment a test function $\psi \in C_c^\infty(B_2;\mb R^3)$ which satisfies also the symmetry condition $\psi \circ g=R\circ \psi$. Extend $\psi$ by $0$ to $\mb R^3$. Let $Q_t(x) \in SO(3)$ be the smooth field of rotations which is determined through the equations
        \[
        \partial_t(Q_t(x)a)=\psi(x)\times Q_t(x)a,\quad Q_0(x)=\Id,\quad \text{for all }a \in \mb R^3.
        \]
        For every $t \in \mb R$, we have $Q_tv_n \in \mc C$. Indeed, it is simple to see that $Q_t v_n$ has the desired symmetry of the admissible class $\mc C$ and equals the constant $e_3$ outside $B_2$. Moreover, $Q_tv_n - e_3 \in W^{1,p}(\mb R^3;\mb R^3)$ and by Lemma \ref{lem:RotPreserveDiv}, the prescribed singularities are preserved.\\
        Thus,
        \begin{equation}
            \label{eq:RotAdmComp}
            E_{B_2}(Q_t v_n)- E_{B_2}(v_n) \geq m - E_{B_2}(v_n).
        \end{equation}
        Moreover, orthogonality of $Q_t$ gives
        \begin{equation}
            \label{eq:EnergyDiff}
            E_{B_2}(Q_t v_n)- E_{B_2}(v_n)=2\sum_{i=1}^3\int_{B_2}\partial_i v_n\cdot (Q_t^T(\partial_i Q_t)v_n)dx+\sum_{i=1}^3\int_{B_2}|(\partial_i Q_t)v_n|^2dx.
        \end{equation}
        From \eqref{eq:ConvMin}, we have $Q_t^T(\partial_i Q_t)v_n \to Q_t^T(\partial_i Q_t)u$ and $(\partial_i Q_t)v_n \to (\partial_i Q_t)u$ strongly in $L^2(B_2)$. Moreover, $\partial_i v_n \rightharpoonup \partial_i u$ weakly. This convergence is enough to pass to the limit in \eqref{eq:EnergyDiff}. Since this energy formula of course also holds for $u$, we have together with \eqref{eq:RotAdmComp} that
        \[
        E_{B_2}(Q_t u)-E_{B_2}(u)=\lim_{n \to \infty}(E_{B_2}(Q_t v_n)-E_{B_2}(v_n)) \geq \lim_{n \to \infty}(m-E_{B_2}(v_n))=0.
        \]
        Since $\p_t (Q_t u)|_{t=0}=\psi \times u$, differentiating at $t=0$ gives
        \[
        0=\frac{d}{dt}\bigg|_{t=0}E_{B_2}(Q_t u)=2\int_{B_2}\sum_{i=1}^3 (u \times \partial_i u)\cdot \partial_i \psi dx,
        \]
        so
        \begin{equation}
            \label{eq:IntHarmSymm}
            \int_{B_2}\sum_{i=1}^3 (u \times \partial_i u)\cdot \partial_i \psi dx=0,
        \end{equation}
        for any $\psi \in C_c^\infty(B_2;\mb R^3)$, which satisfies $\psi \circ g=R\circ \psi$. This identity holds for arbitrary $\psi \in C_c^\infty(B_2;\mb R^3)$ as well: First note that $u$ satisfies $u \circ g=R \circ u$ almost everywhere from the convergence \eqref{eq:ConvMin}. If we write $I(\psi)$ for the integral in \eqref{eq:IntHarmSymm}, change of variables and the symmetry of $u$ give
        \[
        I(R^{-1}\psi \circ g)=I(\psi).
        \]
        Hence, by linearity of $I(\psi)$ in $\psi$
        \[
        I(\psi)=I\left(\frac{\psi+R^{-1}\psi \circ g}{2}\right)=0,
        \]
        where we used that $(\psi+R^{-1}\psi \circ g)/2 \in C_c^\infty(B_2;\mb R^3)$ is a symmetric test function. Thus, we have shown that
        \[
        \int_{B_2}\sum_{i=1}^3 (u \times \partial_i u)\cdot \partial_i \psi dx=0,\quad \forall \psi \in C_c^\infty(B_2;\mb R^3).
        \]
        It is well known that this is equivalent to weak harmonicity of $u$; see, for example, \cite[Section~1.3]{Helein2002}.\\
        Finally, if we assume that $u$ is constant, its trace would give that $u=e_3$. But then we have from \eqref{eq:ConvMin}, Proposition \ref{prop:Constant-expensive} and \eqref{eq:mSmall} that
        \[
        32\pi \leq \liminf_{n \to \infty} E_{B_2}(v_n)=m < 32\pi,
        \]
        which is the desired contradiction. This completes the proof.
    \end{proof}

    \begin{remark}
        Note that our proof does not give any information on the singular set of $u$, except that it is nonempty. It would be interesting to understand if $u$ for example retains the prescribed singularities from the class $\mc C$.
    \end{remark}

    \noindent\textbf{AI Disclosure.}
    The author acknowledges the use of AI tools in the preparation of this work. The author asked ChatGPT 5.6 Sol to construct a competitor with the desired properties in Proposition \ref{prop:trial} and Proposition \ref{prop:gap}. The AI suggested a more complicated competitor and gave proofs of the desired properties. The author simplified the competitor and substantially rewrote the proofs. This paper was written and checked by the author, who takes full responsibility for its correctness.

	\let\oldthebibliography\thebibliography
	\let\endoldthebibliography\endthebibliography
	\renewenvironment{thebibliography}[1]{
	\begin{oldthebibliography}{#1}
	\setlength{\itemsep}{0.5pt}
	\setlength{\parskip}{0.5pt}
	}
	{
	\end{oldthebibliography}
	}
	
	{\small
	\IfFileExists{abbrv-dz.bst}{\bibliographystyle{abbrv-dz}}{\bibliographystyle{abbrv}}
	\bibliography{sources}

@article{Pakzad2001,
  author  = {Pakzad, Mohammad Reza},
  title   = {Existence of infinitely many weakly harmonic maps
             from a domain in {$\mathbb{R}^n$} into {$S^2$}
             for non-constant boundary data},
  journal = {Calculus of Variations and Partial Differential Equations},
  volume  = {13},
  pages   = {97--121},
  year    = {2001},
  doi     = {10.1007/PL00009925}
}

@article{ever,
  author  = {Rivi\`ere, Tristan},
  title   = {Everywhere discontinuous harmonic maps into spheres},
  journal = {Acta Mathematica},
  volume  = {175},
  number  = {2},
  year    = {1995},
  pages   = {197--226},
  doi     = {10.1007/BF02393305}
}

@incollection{BBC,
  author    = {Bethuel, Fabrice and Brezis, Ha{\"i}m and Coron, Jean-Michel},
  title     = {Relaxed energies for harmonic maps},
  booktitle = {Variational Methods: Proceedings of a Conference, Paris, June 1988},
  editor    = {Berestycki, Henri and Coron, Jean-Michel and Ekeland, Ivar},
  series    = {Progress in Nonlinear Differential Equations and Their Applications},
  volume    = {4},
  publisher = {Birkh{\"a}user},
  address   = {Boston, MA},
  year      = {1990},
  pages     = {37--52},
  doi       = {10.1007/978-1-4757-1080-9_3}
}

@article{HLP,
  author  = {Hardt, Robert and Lin, Fang-Hua and Poon, Chi-Cheung},
  title   = {Axially symmetric harmonic maps minimizing a relaxed energy},
  journal = {Communications on Pure and Applied Mathematics},
  volume  = {45},
  number  = {4},
  year    = {1992},
  pages   = {417--459},
  doi     = {10.1002/cpa.3160450404}
}

@article{Rline,
  author  = {Rivi\`ere, Tristan},
  title   = {Applications harmoniques de {$B^3$} dans {$\mb S^2$} ayant une ligne de singularit{\'e}s},
  journal = {Comptes Rendus de l'Acad{\'e}mie des Sciences. S{\'e}rie I. Math{\'e}matique},
  volume  = {313},
  number  = {9},
  year    = {1991},
  pages   = {583--587}
}

@article{BCL,
  author  = {Brezis, Ha{\"i}m and Coron, Jean-Michel and Lieb, Elliott H.},
  title   = {Harmonic maps with defects},
  journal = {Communications in Mathematical Physics},
  volume  = {107},
  number  = {4},
  year    = {1986},
  pages   = {649--705},
  doi     = {10.1007/BF01205490}
}

@book{Helein2002,
  author    = {H{\'e}lein, Fr{\'e}d{\'e}ric},
  title     = {Harmonic Maps, Conservation Laws and Moving Frames},
  series    = {Cambridge Tracts in Mathematics},
  volume    = {150},
  edition   = {Second},
  publisher = {Cambridge University Press},
  address   = {Cambridge},
  year      = {2002},
  doi       = {10.1017/CBO9780511543036},
  isbn      = {978-0-521-81160-6}
}

@article{KarcherWood1984,
  author  = {Karcher, Hermann and Wood, John C.},
  title   = {Non-existence results and growth properties for harmonic maps and forms},
  journal = {Journal f{\"u}r die reine und angewandte Mathematik},
  volume  = {353},
  year    = {1984},
  pages   = {165--180},
  url     = {https://eudml.org/doc/152670}
}

@article{GMS1989,
  author  = {Giaquinta, M. and Modica, G. and Sou{\v{c}}ek, J.},
  title   = {Cartesian currents and variational problems for mappings into spheres},
  journal = {Annali della Scuola Normale Superiore di Pisa, Classe di Scienze (4)},
  volume  = {16},
  number  = {3},
  year    = {1989},
  pages   = {393--485},
  url     = {https://numdam.org/item/ASNSP_1989_4_16_3_393_0/}
}

@article{ShenYan1993,
  author  = {Shen, Yaotian and Yan, Shusen},
  title   = {Constant boundary value problem for a quasilinear elliptic system},
  journal = {Manuscripta Mathematica},
  volume  = {79},
  number  = {1},
  year    = {1993},
  pages   = {99--112},
  doi     = {10.1007/BF02568331}
}

@article{ChouZhu1995,
  author  = {Chou, Kai Seng and Zhu, Xi-Ping},
  title   = {Some constancy results for nematic liquid crystals and harmonic maps},
  journal = {Annales de l'Institut Henri Poincar{\'e} C, Analyse non lin{\'e}aire},
  volume  = {12},
  number  = {1},
  year    = {1995},
  pages   = {99--115},
  doi     = {10.1016/S0294-1449(16)30169-X}
}

@article{BethuelBrezis1991,
  author  = {Bethuel, Fabrice and Brezis, Ha{\"i}m},
  title   = {Regularity of minimizers of relaxed problems for harmonic maps},
  journal = {Journal of Functional Analysis},
  volume  = {101},
  number  = {1},
  year    = {1991},
  pages   = {145--161},
  doi     = {10.1016/0022-1236(91)90152-U}
}
	}

\end{document}